\documentclass[11pt,reqno,twoside]{amsart}
\usepackage{graphicx}
\UseRawInputEncoding 
\usepackage{subfigure}
\usepackage{amsmath,mathdots,amssymb,latexsym,amsthm,mathrsfs,epsfig}
\usepackage{multirow}
\usepackage{breqn}
\usepackage{float}
\usepackage{cite}
\usepackage[colorlinks=true, linkcolor=blue, citecolor=blue]{hyperref}
\usepackage{enumitem}
\usepackage{environ}
\usepackage{mathdots}
\usepackage{mathtools}
\NewEnviron{myequation}{
	\begin{equation}
		\scalebox{1.1}{$\BODY$}
	\end{equation}
}
\usepackage{color}
\usepackage{soul}
\renewcommand{\baselinestretch}{1.2}
\makeatletter
\newcommand{\single}{\let\CS=\@currsize\renewcommand{\baselinestretch}{1.1}\tiny\CS}
\newcommand{\singb}{\let\CS=\@currsize\renewcommand{\baselinestretch}{1}\tiny\CS}
\newcommand{\singa}{\let\CS=\@currsize\renewcommand{\baselinestretch}{1.2}\tiny\CS}
\newcommand{\oneandahalfspacing}{\let\CS=\@currsize\renewcommand{\baselinestretch}{1.5}\tiny\CS}
\newcommand{\singlespacing}{\let\CS=\@currsize\renewcommand{\baselinestretch}{1.6}\large\CS}
\newcommand{\bc}{\begin{center}}
	\newcommand{\ec}{\end{center}}
\newcommand{\be}{\begin{eqnarray}}
	\newcommand{\ee}{\end{eqnarray}}

\makeatletter
\newcommand{\beano}{\begin{eqnarray*}}
	\newcommand{\eeano}{\end{eqnarray*}}

\newcommand{\ba}{\begin{array}}
	\newcommand{\ea}{\end{array}}

\makeatletter

\newcommand{\eval}[2][\right]{_lax
	\ifx#1\right\relax \left.\fi#2#1\rvert}
\newcommand{\diag}{\operatorname{diag}}

\newcommand{\Ind}{\operatorname{Ind}}

\newtheorem{remark}{Remark}[section]

\newtheorem{definition}{Definition}[section]

\newtheorem{theorem}{Theorem}[section]

\newtheorem{lemma}{Lemma}[section]
\numberwithin{equation}{section}
\allowdisplaybreaks
\begin{document}
	
	\title[Symplectic model for Zelevinsky modules of $\mathrm{GL}(n,\mathrm{D})$]{Symplectic model for Zelevinsky modules of $\mathrm{GL}(n,\mathrm{D})$}
	

\author[Hariom Sharma]{Hariom Sharma}
	\address{Department of Mathematics, Indian Institute of Technology Bombay, Mumbai, 400076, Maharashtra, India}
	\email{hariomshrma97@gmail.com, hariom@math.iitb.ac.in}

	
%


	\begin{abstract} 
		Let $\mathrm{D}$ be a quaternion division algebra over a non-archimedean local field $\mathrm{F}$ of characteristic zero. This article demonstrates the existence and uniqueness of the symplectic model for a family of Zelevinsky modules of $\mathrm{GL}(n,\mathrm{D})$ to a family of irreducible representations of $\mathrm{GL}(n,\mathrm{D})$. For this family of irreducible representations, we identify a necessary condition under which a symplectic model can exist. This  work extends a result of Offen and Sayag beyond the case $\mathrm{D} = \mathrm{F}$.
	\end{abstract}
	
	\maketitle
	

	\setcounter{page}{1}
\section{Introduction}\label{sec1}
Let $\mathrm{F}$ be a non-Archimedean local field of characteristic zero, and let $\mathrm{D}$ be a quaternion division algebra with center $\mathrm{F}$.  
For any integer $n \geq 1$, set $G_n = \mathrm{GL}(n,\mathrm{D})$.
We then set  
\[
H_n = \{\, g \in G_n \mid {}^{t}\!\bar{g} J_n g = J_n \,\},
\]
where the map $x \mapsto \bar{x}$ denotes the main involution on $\mathrm{D}$, applied entrywise to elements of $G_n$, and $J_n$ is the anti-diagonal matrix in $G_n$ given by
$$
J_n =
\begin{pmatrix}
	& & 1 \\
	&  \iddots & \\
	1 & &
\end{pmatrix}.
$$
Note that $G_n$ and $H_n$ are inner forms of the general linear group $\mathrm{GL}(2n, \mathrm{F})$ and the symplectic group $\mathrm{Sp}(2n, \mathrm{F})$, respectively. Throughout, a representation of $G_n$ will always mean a smooth admissible representation. Denote by $\mathrm{Irr}(G_n)$ the set of equivalence classes of irreducible representations of $G_n$. A representation $(\pi,V)$ of $G_n$ is said to have a symplectic model (or to be $H_n$-distinguished) if $\mathrm{Hom}_{H_n}(\pi, \mathbb{C}) \not= 0$. 
Verma \cite[Theorem $1.1$]{Verma} proved the uniqueness of the symplectic model for irreducible representations, i.e., 
\begin{lemma}\label{10}
	For $\pi \in \mathrm{Irr}(G_n)$, we have 
	$\mathrm{dim} \mathrm{Hom}_{H_n}(\pi, \mathbb{C}) \leq 1$.
\end{lemma}
This article aims to prove the existence and uniqueness of the symplectic model for Zelevinsky modules of $G_n$ (see Section \ref{199}). In \cite[Section $4$]{Verma}, Verma also classified the irreducible representations of \(G_1\) and \(G_2\) admitting a symplectic model. 
Sharma and Verma \cite[Theorems $1.2$ and $1.3$]{sharma2023symplectic} later extended this classification to \(G_3\) and \(G_4\), and further to a family of ladder and unitary representations \cite[Theorems $1.1$ and $1.3$]{SV2024} for general \(n\). 
The classification of discrete series  representations of \(G_n\) admitting a symplectic model was recently obtained by Matringe et.al. \cite[Theorems $1.2$ and $1.3$]{MSSS26}. 
In the global setting, the symplectic period for the pair $(G_2,H_2)$ was studied by Verma   \cite{verma2018distinguished}. We are currently investigating the global symplectic period for the pair $(G_{2n}, H_{2n})$.


\section{Main result}\label{199}
We begin by introducing some notation and recalling the Zelevinsky classification \cite{minguez2013representations}, which will be used in the formulation of our results. 
Let $
\mathrm{Irr} = \bigsqcup_{n=0}^{\infty} \mathrm{Irr}(G_n),$
where $\mathrm{Irr}(G_0)$ consists of the trivial representation of the trivial group.   Let $\mathrm{Cusp}$ be a subset of supercuspidal representations in $\mathrm{Irr}$.
Let $\nu$ denote the character of $\mathrm{G}_n$ obtained by composing the normalized absolute value on $\mathrm{F}$ with the reduced norm, for any $n \in \mathbb{N}$.
A Zelevinsky segment is a set of the form $[a,b]_{(\rho)} = \{ \nu^{l_{\rho} i} \rho : i = a, a+1, \dots, b \},$
where $\rho \in \mathrm{Irr}(G_d)$ is supercuspidal for some $d \in \mathbb{N}$, $l_{\rho}$ is a positive integer associated with $\rho$, and $a,b \in \mathbb{Z}$ satisfy $a-1 \le b$ (with $b = a-1$ corresponding to the empty segment).  
Note that if $\rho$ is an irreducible supercuspidal representation of $G_1$, then $l_{\rho}=2$ when $\dim(\rho)=1$ and $l_{\rho}=1$ otherwise; on the other hand, if $\rho$ is an irreducible supercuspidal representation of $G_2$, then $l_{\rho}=1$.
A detailed description of $l_{\rho}$ can be found in \cite[Section $2$]{Tad90}. 

A multi-set of Zelevinsky segments is a finitely supported function
\[
\mathfrak{m} : \{\text{Zelevinsky segments}\} \longrightarrow \mathbb{Z}_{\ge 0}.
\]
Equivalently, we may represent $\mathfrak{m}$ as an unordered tuple $\mathfrak{m} = \{\Delta_1, \dots, \Delta_t\},$ where $\mathfrak{m}(\Delta) = \#\{\, i : \Delta_i = \Delta \,\}$
denotes the multiplicity of each segment $\Delta$. When a particular order is fixed, by abuse of notation, we write $\mathfrak{m} = (\Delta_1, \dots, \Delta_t)$. There exists a bijection
$\mathfrak{m} \longmapsto Z(\mathfrak{m})$
between the set of multi-sets of Zelevinsky segments and the set $\mathrm{Irr}$ of irreducible representations of all general linear groups over $\mathrm{D}$, whose description is made explicit below.

For a segment $\Delta = [a,b]_{(\rho)}$, let $Z(\Delta)$ denote the unique irreducible subrepresentation of
$\nu^a \rho \times \cdots \times \nu^b \rho,$ which is also the unique irreducible quotient of $\nu^b \rho \times \cdots \times \nu^a \rho.$
Here, $\times$ denotes the Bernstein--Zelevinsky product (normalized parabolic induction), and $Z(\emptyset)$ is the trivial representation of the trivial group.

Two Zelevinsky segments $\Delta$ and $\Delta'$ are called linked if neither is contained in the other and their union is again a segment. If this occurs, we may express
\[
\Delta = [a,b]_{(\rho)}, \quad \Delta' = [a',b']_{(\rho)}
\]
for some supercuspidal representation $\rho \in \mathrm{Irr}$ and integers $a,b,a',b' \in \mathbb{Z}$.  
In this case, we say that $\Delta$ precedes $\Delta'$ whenever $a < a'$, $b < b'$, and $b \geq a'-1$. 

For the purposes of this article, a multi-set $\mathfrak{m} = (\Delta_1, \dots, \Delta_t)$ of Zelevinsky segments is called well-ordered if $\Delta_j$ does not precede $\Delta_i$ whenever $1 \leq i < j \leq t$, and the associated representation
\[
\pi(\mathfrak{m}) = Z(\Delta_1) \times \cdots \times Z(\Delta_t)
\]
is then independent of the chosen ordering and has a unique irreducible quotient $Z(\mathfrak{m})$, establishing the bijection in the Zelevinsky classification. Throughout this article, we refer to $\pi(\mathfrak{m})$ as a \textbf{Zelevinsky module}.

We now introduce the notion of cuspidal support for a representation $\pi$ of $G_n$.  
For $\pi \in \mathrm{Irr}$, there exist $\rho_1, \ldots, \rho_k \in \mathrm{Cusp}$, unique up to permutation, such that $\pi$ embeds as a subrepresentation of $\rho_1 \times \cdots \times \rho_k$.
The cuspidal support of $\pi$ is defined as
\[
\mathrm{Supp}(\pi) := \{\rho_1, \ldots, \rho_k\}.
\]
If $\sigma_1, \ldots, \sigma_k \in \mathrm{Irr}$, then we set
$\mathrm{Supp}(\sigma_1 \otimes \cdots \otimes \sigma_k) := \bigcup_{i=1}^k \mathrm{Supp}(\sigma_i)$.
Let $M$ be a standard Levi subgroup of $G$ and let  $\pi$ be a representation of $M$ such that  $\pi_1, \ldots, \pi_t$ are the irreducible components of $\pi$. Then we set
\[
\mathrm{Supp}(\pi) := \bigcup_{i=1}^t \mathrm{Supp}(\pi_i).
\]

Moving forward, we introduce the weighted odd part function $\mathfrak{O}$ on multi-sets of Zelevinsky segments. Let $\rho \in \mathrm{Irr}(G_r)$ for some $r \in \mathbb{N}$, and suppose $\rho \in \Delta$; then set $d(\Delta) = r$. For a Zelevinsky segment $\Delta$, let $\# (\Delta)$ be the cardinality of $\Delta$, and set
\[
\mathfrak{O}(\Delta) =
\begin{cases}
	d(\Delta), & \text{if } \#(\Delta) \text{ is odd},\\[1mm]
	0, & \text{if } \#(\Delta) \text{ is even}.
\end{cases}
\]
For a multi-set $\mathfrak{m} = \{\Delta_1, \dots, \Delta_t\}$, we set
\[
\mathfrak{O}(\mathfrak{m}) = \sum_{i=1}^t \mathfrak{O}(\Delta_i).
\]

The main result of this article is the following.
\begin{theorem}\label{1}
	Let	$\mathfrak{m} = \{\Delta_1, \ldots, \Delta_k\} \quad \text{with} \quad 
	\Delta_i = [a_i, b_i]_{(\rho_i)},
	$
	where each \(\rho_i \in \mathrm{Irr}(G_d)\) is a supercuspidal representation with no symplectic model.
	Let $\pi(\mathfrak{m})$ be a representation of $G_n$. Then $\pi(\mathfrak{m})$ has a symplectic model if and only if $\mathfrak{O}(\mathfrak{m}) = 0$, and  in this case,
	$$\mathrm{dim} \mathrm{Hom}_{H_n}(\pi(\mathfrak{m}), \mathbb{C})= 1.$$
	Furthermore, if $\pi = Z(\mathfrak{m}) \in \mathrm{Irr}$ admits a symplectic model, then $\mathfrak{O}(\mathfrak{m}) = 0$.
\end{theorem}

When $\mathrm{D} = \mathrm{F}$, the above result is a theorem due to Offen and Sayag \cite[Theorem $2.1$]{OS19}. 
The methods and proof ideas in this article are inspired by \cite{OS19} and \cite{mitra2017klyachko}.
We follow their approach in developing the lemmas and the main result.

\begin{remark}
	Let $\rho$ be an irreducible supercuspidal representation of $G_2$, and consider the two segments of length one on the same $\rho$-line:
	\[
	\Delta_{1} = [-1]_{(\rho)} = \nu^{-1}\rho, \qquad 
	\Delta_{2} = [0]_{(\rho)} = \rho.
	\]
	These segments are naturally ordered by their exponents, with $-1 < 0$. 
	The product
	\[
	\pi(\mathfrak{m}) = Z(\Delta_{1}) \times Z(\Delta_{2}) = \nu^{-1}\rho \times \rho
	\]
	is the Zelevinsky module corresponding to the multiset $\mathfrak{m} = \{\Delta_{1}, \Delta_{2}\}$. 
	Since both segments have odd length, the parameter $\mathfrak{O}(\mathfrak{m})$ is nonzero. 
	According to Theorem~\ref{1}, this means that $\pi(\mathfrak{m})$ does not have a symplectic model.
	On the other hand, the representation $Z([-1,0]_{(\rho)})$ associated with the single segment of even length satisfies $\mathfrak{O}(\mathfrak{m}) = 0$. 
	By Theorem~\ref{1}, this representation has a unique symplectic model.
	
	The representation $\rho \times \nu^{-1}\rho$ reverses the natural order of exponents, placing $0$ before $-1$. 
	Therefore, it is not a Zelevinsky module in the sense of Theorem~\ref{1}. 
	Instead, it is the Langlands standard module associated with the same data, whose irreducible quotient is $Z([-1,0]_{(\rho)})$. 
	Thus, Theorem~\ref{1} does not apply directly to $\rho \times \nu^{-1}\rho$. 
	However, since its quotient $Z([-1,0]_{(\rho)})$ has a symplectic model, the standard module $\rho \times \nu^{-1}\rho$ also inherits a symplectic model through the quotient map.
\end{remark}

	\section{Some useful lemmas}\label{sec2}
Let us recall that $\nu$ denotes the character of $\mathrm{G}_n$ obtained by composing the normalized absolute value on $\mathrm{F}$ with the reduced norm, for any $n \in \mathbb{N}$. 
Let $\mathcal{E}$ be the graph with vertex set $\mathrm{Cusp}$; for every $\rho \in \mathrm{Cusp}$, an edge is drawn between $\rho$ and $\nu^{l_{\rho}}\rho$. 
For each finite subset $\mathfrak{V} \subseteq \mathrm{Cusp}$, denote by 
$\mathcal{E}_{\mathfrak{V}}$ the subgraph induced on the vertex set $\mathfrak{V}$, 
and by $\mathcal{M}_{\mathfrak{V}}$ the set of its connected components. 
Each connected component $\Delta \in \mathcal{M}_{\mathfrak{V}}$ has the form
$\Delta = \{ \nu^{l_{\rho} i} \rho : i = a, \ldots, b \}$
for some $\rho \in \mathrm{Cusp}$ and integers $a \leq b$.
\begin{definition}{\upshape(%
		\textbf{Mitra et. al.}, \cite[Definition $5.7$]{mitra2017klyachko})}
	Two finite subsets $\mathfrak{V}_1, \mathfrak{V}_2 \subseteq \mathrm{Cusp}$ are said to be 
	totally disjoint if either they belong to disjoint supercuspidal $\mathbb{Z}$-lines, 
	or the following holds:  
	for every $\Delta_1 \in \mathcal{M}_{\mathfrak{V}_1}$ and 
	$\Delta_2 \in \mathcal{M}_{\mathfrak{V}_2}$, one of the two inequalities is satisfied:
	$
	\nu^{l_{\rho_1}}\rho_1 < \rho_2 
	\quad \text{for all } \rho_1 \in \Delta_1, \ \rho_2 \in \Delta_2,
	$
	or
	$\nu^{l_{\rho_2}}\rho_2 < \rho_1 
	\quad \text{for all } \rho_1 \in \Delta_1, \ \rho_2 \in \Delta_2$.
\end{definition}
The next lemma is obtained using an approach analogous to that of \cite[Lemma $5.9$]{mitra2017klyachko}.
\begin{lemma}\label{299}
	Let $\pi_1, \ldots, \pi_k$ be representations of $G_{n_i}$ whose supports 
	$\mathrm{Supp}(\pi_i)$ and $\mathrm{Supp}(\pi_j)$ are totally disjoint for all distinct 
	$i, j$. Assume that each supercuspidal representation in $\mathrm{Supp}(\pi_i)$ does not admit a symplectic model for each $i=1,\ldots,k$.  If $\pi = \pi_1 \times \cdots \times \pi_k$
	admits a symplectic model, then so does each $\pi_i$, for every $i = 1, \ldots, k$. 
\end{lemma}
\begin{proof}
	Let $\sigma = \pi_1 \otimes \cdots \otimes \pi_k$, and set $n = n_1 + \cdots + n_k$.
	Then $\pi = \mathrm{Ind}_P^{G_n}(\sigma)$ is a representation of $G_n$, where $P=MU$ is the parabolic subgroup of $G_n$ corresponding to the partition $(n_1, \ldots, n_k)$ (with \(\Ind\) indicating normalized parabolic induction). 
	By \cite[Lemma 4.1]{SV2024}, there exists a \(P\)-orbit in \(X\) 
	serving as a good orbit for \(\sigma\) (see \cite[Definition 4.1]{SV2024}). 
	Following the notation of \cite[$\S$3.5]{SV2024}, one can easily see that there exists an irreducible component 
	$\rho$ of $\mathrm{r}_{L,M}(\sigma)$ (where $\mathrm{r}_{L,M}$ denotes the normalized Jacquet module) 
	satisfying the conditions of \cite[Proposition 4.2]{SV2024}.  
	Moreover, we have
	$\rho = \bigotimes_{\mathfrak{i} \in (\mathfrak{J}, \prec)} \rho_i,$
	where
	$\rho_{i,1} \otimes \cdots \otimes \rho_{i,k_i}$ is an irreducible component of 
	$\mathrm{r}_{M_{\beta_i}, G_{n_i}}(\pi_i)$, for each $i = 1, \ldots, k$.  
	In particular, one has
	\[
	\mathrm{Supp}(\rho_{i,j}) \subseteq \mathrm{Supp}(\mathrm{r}_{M_{\beta_i}, G_{n_i}}(\pi_i)) \subseteq \mathrm{Supp}(\pi_i) \quad \text{for all } i \quad (\text{see  \cite[\S 2.12]{BZ})}.
	\]
	
	Suppose there exists $\mathfrak{i} \in \mathfrak{J}$ with $\mathfrak{i} \prec \tau(\mathfrak{i})$, and write $\mathfrak{i} = (i_1,j_1)$ and $\tau(\mathfrak{i}) = (i_2,j_2)$.
	Then we have $\rho_{\mathfrak{i}} \simeq \nu \rho_{\tau(\mathfrak{i})}$, where $i_1 \neq i_2$.
	In particular, there exists $\rho_2 \in \mathrm{Supp}(\rho_{\tau(\mathfrak{i})}) \subseteq \mathrm{Supp}(\pi_{i_2})$ such that $\nu \rho_2 \in \mathrm{Supp}(\rho_{\mathfrak{i}}) \subseteq \mathrm{Supp}(\pi_{i_1})$, which contradicts the total disjointness of $\mathrm{Supp}(\pi_{i_1})$ and $\mathrm{Supp}(\pi_{i_2})$.
	Hence, $\tau$ must be the trivial involution. It then follows from \cite[Proposition 4.2]{SV2024} that each $\pi_i$ admits a symplectic model for $i=1, \ldots, k$, as required.
\end{proof}

The result below follows as a particular case of \cite[Theorem $1.1$]{SV2024}. 
\begin{lemma}\label{11}
	Let $\rho \in \mathrm{Irr}(G_d)$ be a supercuspidal representation with no symplectic model. Then, for a Zelevinsky segment $\Delta = [a,b]_{(\rho)}$, $\mathrm{Z}(\Delta)$ admits a symplectic model if and only if $\mathfrak{O}(\Delta) = 0$.
\end{lemma}
Let us further recall the hereditary property of the symplectic model \cite[Theorem $1.2$]{SV2024}.
\begin{lemma}\label{12}
	Let $\pi_i$ be a representation of $G_{n_i}$ of finite length that admits a symplectic model for each $i = 1,\ldots,k$. Then $\pi_1 \times \cdots \times \pi_k$ admits a symplectic model.  
\end{lemma}
As a direct consequence of Lemmas \ref{11} and \ref{12}, we obtain the following lemma. 
\begin{lemma}\label{13}
	Let $\mathfrak{m} = \{\Delta_1, \ldots, \Delta_t\}$ be a multi-set of Zelevinsky segments (independently of the order chosen), where $\Delta_i = [a_i, b_i]_{(\rho_i)},$
	and each \(\rho_i \in \mathrm{Irr}(G_d)\) is a supercuspidal representation with no symplectic model.  If $\mathfrak{O}(\mathfrak{m}) = 0$, then the representation 
	\[
	Z(\Delta_1) \times \cdots \times Z(\Delta_t)
	\]
	admits a symplectic model.
\end{lemma}
The statement of Lemma~\ref{11} can be extended to products of \( Z(\Delta) \)'s. However, before proceeding, we must first describe their Jacquet modules explicitly. We begin with the Jacquet module associated with a single  \( Z(\Delta) \). 

Let $\Delta = [a,b]_{(\rho)}$ be a segment in $\mathrm{Cusp}$. Let us recall the description of the Jacquet 
module of $Z(\Delta)$. 
Assume $\rho$ be the representation of $G_d$ and set $n = (b - a + 1)d$, so that 
$Z(\Delta)$ is the representation of $G_n$. Let $M_\alpha$ be the Levi subgroup of $G_n$ corresponding to the partition $\alpha = (n_1, \ldots, n_k)$ of $n$. Then
\[
\mathrm{r}_{M_{\alpha},G_n}(Z(\Delta)) = 0 
\quad \text{unless } d \mid n_i \text{ for all } i = 1, \ldots, k,
\]
in which case we have
\[
\mathrm{r}_{M_{\alpha},G_n}(Z(\Delta)) = Z(\Delta_1) \otimes \cdots \otimes Z(\Delta_k),
\]
where $\Delta_i = [a_i, b_i]_{(\rho)}$ with $a_1 = a$, 
$a_{i+1} = b_i + 1$ for $i = 1, \ldots, k - 1$, and 
$d(b_i - a_i + 1) = n_i$ for each $i = 1, \ldots, k$.

Let $\beta$ be a refinement of a partition $\alpha$ of $n$, and denote 
$M = M_\alpha$ and $L = M_\beta$. 
For the parts of these compositions and for the ordered index set $I$, 
we use the notation introduced in \cite[ $\S$$3.5$]{SV2024}.

Consider segments $\Delta_1, \ldots, \Delta_k$ of supercuspidal representations such that 
\[
\sigma = Z(\Delta_1) \otimes \cdots \otimes Z(\Delta_k)
\]
is an irreducible representation of $M$. 
Then, according to the above description of the Jacquet module of $Z(\Delta)$, whenever the Jacquet module does not vanish, we have
\[
\mathrm{r}_{L,M}(\sigma) = \bigotimes_{\mathfrak{i} \in (\mathfrak{J}, \prec)} Z(\Delta_{\mathfrak{i}}),
\]
where each
\[
\mathrm{r}_{M_{\beta_i}, G_{n_i}}(Z(\Delta_i)) 
= Z(\Delta_{i,1}) \otimes \cdots \otimes Z(\Delta_{i,k_i})
\]
is determined as above.  The assertion of Lemma~\ref{11} extends naturally to tensor products of $Z(\Delta)$'s, whose proof relies on the same idea as in \cite[Proposition $7.5$]{mitra2017klyachko}.


\begin{lemma}\label{16}
	Let	$\mathfrak{m} = \{\Delta_1, \ldots, \Delta_k\} \quad \text{with} \quad 
	\Delta_i = [a_i, b_i]_{(\rho)},
	$
	where \(\rho \in \mathrm{Irr}(G_d)\) is a supercuspidal representation with no symplectic model and all \(a_i, b_i \in \mathbb{Z}\). 
	If $\pi(\mathfrak{m})$ admits a symplectic model, then $\mathfrak{O}(\mathfrak{m}) = 0$.
\end{lemma}
\begin{proof}
	Let \(\mathfrak{m} = \{\Delta_1, \ldots, \Delta_k\}\) be ordered so that 
	\(a_1 \leq \cdots \leq a_k\), and note that the reverse order 
	\(\{\Delta_k, \ldots, \Delta_1\}\) is standard. 
	Define 
	\(\sigma = Z(\Delta_1) \otimes \cdots \otimes Z(\Delta_k)\) 
	so that
	\[
	Z(\Delta_1) \times \cdots \times Z(\Delta_k) 
	= \mathrm{Ind}_P^{G_n}(\sigma),
	\]
	with \(P\) a standard parabolic of \(G_n\). 
	By \cite[Lemma 4.1]{SV2024}, there exists a \(P\)-orbit in \(X\) 
	serving as a good orbit for \(\sigma\). 
	Suppose, by contradiction, that some \(Z(\Delta_i)\) does not admit a symplectic model. 
	Using the notation of \cite[ $\S$$3.5$]{SV2024},  it is easy to see by \cite[Proposition 4.2]{SV2024} 
	that \(\tau\) is nontrivial. 
	Let \(\mathfrak{i} \in \mathfrak{J}\) be minimal with \(\tau(\mathfrak{i}) \neq \mathfrak{i}\). 
	Then \(\mathfrak{i} \prec \tau(\mathfrak{i})\), and by \cite[Eq. (3.3)]{SV2024}, 
	\(\mathfrak{i} = (i,1)\) for some \(1 \leq i \leq k\). 
	The relation 
	\(\Delta_{\mathfrak{i}} = \nu \Delta_{\tau(\mathfrak{i})}\) 
	from \cite[Proposition 4.2]{SV2024} then contradicts the chosen order on \(\mathfrak{m}\) and structure of $\mathrm{r}_{L,M}(\sigma)$ described prior to the lemma.
	Therefore, each \(Z(\Delta_i)\) admits a symplectic model, 
	and by Lemma~\ref{11}, 
	\(\mathfrak{O}(\Delta_i) = 0\) for all \(i = 1, \ldots, k\). 
	Consequently, we have
	$\mathfrak{O}(\mathfrak{m}) = 0$.
\end{proof}
\section{Proof of Theorem \ref{1}}
In this section, we first present the essential tool and then use it to complete the proof of the main result. 

Let $X = \{g \in G_n : g J_n = J_n~ ^t\bar{g}\}$, where  \( J_n = (J_{ij})_{1 \leq i,j \leq n} \) denotes the matrix in \( G_n \) 
whose entries satisfy \( J_{ij} = 1 \) when \( i + j = n + 1 \) and \( 0 \) otherwise.. Then the map $g \mapsto gJ_n~ ^t\bar{g}$ identifies $G_n/H_n$ with $X$.  For a subgroup $Q$ of $G_n$ and $x \in X$, let $Q_x = \{q \in Q : q x~ ^t \bar{q} = x\}$ be the stabilizer of $x$ in $Q$. In particular, we have $H_n = {(G_n)}_{J_n}$.

Let $n= n_1 + \cdots + n_k$ be a partition of $n$, and let 
$M = \{\diag (g_1, \ldots, g_k) : g_i \in G_{n_i},~ i = 1,\ldots, k\}.$
Let $P = MU$ be the block upper triangular parabolic subgroup of $G_n$ with Levi subgroup $M$ and unipotent radical $U$. For $x = \diag (J_{n_1},\ldots, J_k) \in X$, we have 
$$P_x = \{\diag (h_1, \ldots, h_k) : h_i \in H_{n_i},~ i = 1,\ldots, k\} = M_x$$ 
and the $P$ orbit of $x$ is the open $P$-orbit in $X$. 

Choose an element \(\eta \in G_n\) satisfying $\eta J_n \, {}^t\!\bar{\eta} = x.$
Let \(\sigma\) be a representation of \(M\), and denote by 
\(V\) the space of 
\(\pi = \Ind_P^{G_n}(\sigma)\). 
Define \(V_0\) as the \(H_n\)-invariant subspace of sections supported on 
\(P \eta H_n\). 

We say that the symplectic model of \(\pi\) is supported on the open orbit if 
\[
\mathrm{Hom}_{H_n}(V/V_0, \mathbb{C}) = 0.
\]
In this situation, the restriction of \(H_n\)-equivariant functionals to \(V_0\) is injective:
\[
\mathrm{Hom}_{H_n}(V, \mathbb{C}) \hookrightarrow \mathrm{Hom}_{H_n}(V_0, \mathbb{C}).
\]
Moreover, by \cite[Proposition 4.1]{offen2017parabolic}, one has an isomorphism
\[
\mathrm{Hom}_{H_n}(V_0, \mathbb{C}) \simeq \mathrm{Hom}_{M_x}(\sigma, \mathbb{C}).
\]
Suppose further that \(\sigma = \sigma_1 \otimes \cdots \otimes \sigma_k\), where each \(\sigma_i\) is a representation of \(G_{n_i}\) and 
\(\dim \mathrm{Hom}_{H_{n_i}}(\sigma_i, \mathbb{C}) \leq 1\) for \(i = 1, \dots, k\). 
It then follows immediately that
\[
\mathrm{Hom}_{M_x}(\sigma, \mathbb{C}) \simeq \bigotimes_{i=1}^k \mathrm{Hom}_{H_{n_i}}(\sigma_i, \mathbb{C}).
\]
Taken together, these observations yield a key structural result that will be applied twice in what follows.
\begin{lemma}\label{15}
	Let \( n = n_1 + \cdots + n_k \) and let \(\sigma_i\) be a representation of \( G_{n_i} \) such that 
	\(\dim \mathrm{Hom}_{H_{n_i}}(\sigma_i, \mathbb{C}) \leq 1\) for \( i = 1, \ldots, k \). 
	Suppose further that the symplectic model on \( \sigma_1 \times \cdots \times \sigma_k \) is supported on the open orbit. 
	Then there exists a natural embedding
	\[
	\mathrm{Hom}_{H_n}(\sigma_1 \times \cdots \times \sigma_k, \mathbb{C})
	\hookrightarrow 
	\bigotimes_{i=1}^k \mathrm{Hom}_{H_{n_i}}(\sigma_i, \mathbb{C}).
	\]
	In particular,
	\(\dim \mathrm{Hom}_{H_n}(\sigma_1 \times \cdots \times \sigma_k, \mathbb{C}) \leq 1\).
\end{lemma}
We now prove our main result, whose proof follows the argument of  \cite[Lemma $3.9$]{OS19}.
\begin{proof}[Proof of Theorem $\ref{1}$.]
	It is easy to see by Lemma \ref{13} that  the ``if direction" holds. Therefore, if 
	\(\mathfrak{O}(\mathfrak{m}) = 0\), then we also have 
	\[
	\dim \mathrm{Hom}_{H_n}\bigl(\pi(\mathfrak{m}), \mathbb{C}\bigr) \geq 1.
	\]
	The proof of the ``only if direction" and the other inequality is established in two steps. To begin, consider the case where $\mathfrak{m}$ is supported entirely on a single supercuspidal $\mathbb{Z}$-line. Explicitly,
	\[
	\mathfrak{m} = \{\Delta_1, \ldots, \Delta_k\} \quad \text{with} \quad 
	\Delta_i = [a_i, b_i]_{(\rho)},
	\]
	where \(\rho \in \mathrm{Irr}(G_d)\) is supercuspidal and all \(a_i, b_i \in \mathbb{Z}\). The ``only if'' direction in this case is precisely Lemma~\ref{16}. Moreover, the proof of this lemma yields additional information.
	Suppose in addition that \(\mathfrak{m}\) is arranged so that 
	\(a_1 \leq \cdots \leq a_k\).
	In this case,
	\[
	\pi(\mathfrak{m}) = Z(\Delta_1) \times \cdots \times Z(\Delta_k),
	\]
	and Lemma~\ref{16} ensures that its symplectic model is supported on the open orbit.
	By the uniqueness of the symplectic model for irreducible representations (see Lemma \ref{10}), 
	one can then invoke Lemma~\ref{15} with \(\sigma_i = Z(\Delta_i)\). It follows that 
	$\dim \mathrm{Hom}_{H_n}(\pi(\mathfrak{m}), \mathbb{C}) \leq 1$.
	
	Consider now a general multiset \(\mathfrak{m}\), which can be written as
	$\mathfrak{m} = \mathfrak{m}_1 + \cdots + \mathfrak{m}_k$ (viewing multisets as functions with finite support),
	where each \(\mathfrak{m}_i\) is supported on a distinct supercuspidal \(\mathbb{Z}\)-line associated with $\rho_i$. Here, $\rho_i$ does not admit a symplectic model for all $i = 1, \ldots, k$.  
	Concretely, for \(1 \leq i \neq j \leq k\), \(\Delta \in \mathfrak{m}_i\), \(\Delta' \in \mathfrak{m}_j\), 
	\(\rho \in \Delta\), and \(\rho' \in \Delta'\), one has $\rho' \not\simeq \nu^m \rho \quad \text{for all } m \in \mathbb{Z}.$
	It then follows that $$\pi(\mathfrak{m}) = \pi(\mathfrak{m}_1) \times \cdots \times \pi(\mathfrak{m}_k).$$
	Lemma~\ref{299} implies that the symplectic model on
	$\pi(\mathfrak{m}_1) \times \cdots \times \pi(\mathfrak{m}_k)$
	is supported on the open orbit. By invoking Lemma~\ref{15} again with
	\(\sigma_i = \pi(\mathfrak{m}_i)\),
	and using the relation
	\[
	\mathfrak{O}(\mathfrak{m}) = \sum_{i=1}^k \mathfrak{O}(\mathfrak{m}_i),
	\]
	the ``only if direction" follows. In particular, when \(\mathfrak{O}(\mathfrak{m}) = 0\),
	$\dim \mathrm{Hom}_{H_n}\bigl(\pi(\mathfrak{m}), \mathbb{C}\bigr) \leq 1$.
	Let \( n \) be fixed such that \( \pi \in \mathrm{Irr}(G_n) \). By composing with the canonical projection 
	\( \pi(\mathfrak{m}) \rightarrow \pi \), we obtain a natural injective linear map
	\[
	\mathrm{Hom}_{H_n}(\pi, \mathbb{C}) \hookrightarrow \mathrm{Hom}_{H_n}(\pi(\mathfrak{m}), \mathbb{C}) \leq 1.
	\]
	Since $\pi$ admits a symplectic model, the existence of this injection forces the vanishing of the obstruction 
	\( \mathfrak{O}(\mathfrak{m}) \), that is,
	$\mathfrak{O}(\mathfrak{m}) = 0$.
	Consequently, this establishes the desired result and completes the proof of the theorem.
\end{proof}

	\begin{center}
		\large {\textbf{Acknowledgment}}
	\end{center}
	The author  thanks  Dipendra Prasad  for his helpful comments and valuable suggestions. He would like to thank Ravi Raghunathan  for many interesting conversations during the course of the preparation of this article. The author  thanks the Indian Institute of Technology Bombay for supporting this research through the Institute Postdoctoral Fellowship.
	\bibliography{bibliography}
	\bibliographystyle{alpha}
\end{document}